\documentclass[12pt]{amsart}
\usepackage{amsmath,amsthm,amssymb,amsfonts,dsfont}
\usepackage{graphicx}
\usepackage{color}
\usepackage[utf8]{inputenc}

\newtheorem{theorem}{Theorem}[section]
\newtheorem{lemma}[theorem]{Lemma}

\newtheorem{proposition}[theorem]{Proposition}

\newtheorem{corollary}[theorem]{Corollary}
\theoremstyle{definition}
\newtheorem{definition}[theorem]{Definition}

\newtheorem{example}[theorem]{Example}

\newtheorem{remark}[theorem]{Remark}
\numberwithin{equation}{section}

\def\co{\text{\rm co\,}}

\newcommand{\epsi}{\varepsilon}
\newcommand{\eee}{\varepsilon}
\def\spa{\text{\rm span\,}}
\def\llll{\longrightarrow}

\newcommand{\N}{{\mathbb {N}}}

\newcommand{\R}{{\mathbb {R}}}

\def\com#1{{``#1''}}

\def\sep{{ \ \  }}
\def\sem{{\ \ \ \  }}
\def\seg{{\ \ \ \  \ \  }}

\title[Bishop-Phelps-Bollob{\'a}s  property] {Bishop-Phelps-Bollob{\'a}s  property  for positive functionals }

\author[M.D. Acosta]{Mar\'{\i}a D. Acosta}
\address{Universidad de Granada, Facultad de Ciencias,
	Departamento de An\'{a}lisis Matem\'{a}tico, 18071 Granada, Spain}
\email{dacosta@ugr.es}
\author[M. Soleimani]{Maryam Soleimani-Mourchehkhorti}
\address{School of Mathematics, Institute for Research in Fundamental Sciences (IPM), P.O. Box: 19395-5746, Tehran, Iran}
\email{m-soleimani85@ipm.ir}

\thanks{The  first  author was  supported  by Junta de Andaluc\'{\i}a grant  FQM--185,  by Spanish MINECO/FEDER grant PGC2018-093794-B-I00 	
and also by     Junta de Andaluc\'{\i}a grant  A-FQM-484-UGR18. The second author was   supported by a grant from IPM}


\begin{document}


\begin{abstract}
	We introduce the so-called  Bishop-Phelps-Bollob{\'a}s property  for positive functionals, a particular case of the  Bishop-Phelps-Bollob{\'a}s property  for positive operators. First we show a  version of the Bishop-Phelps-Bollob{\'a}s theorem  for positive elements and positive  functionals in the dual of any Banach lattice. We also characterize the strong  Bishop-Phelps-Bollob{\'a}s property  for positive functionals in a Banach lattice.  We prove that any finite-dimensional Banach lattice has the the Bishop-Phelps-Bollob{\'a}s property  for positive functionals.  A sufficient and a necessary condition to have the Bishop-Phelps-Bollob{\'a}s property  for positive functionals are  also provided. As a consequence of this result, we obtain that the spaces $L_p(\mu)$  ($1\le p < \infty$), for any positive measure $\mu,$  $C(K)$ and $\mathcal{M} (K)$,  for any compact and Hausdorff topological space $K,$  satisfy  the Bishop-Phelps-Bollob{\'a}s property  for positive functionals. We also provide some more clarifying examples.
\end{abstract}

\maketitle

\section{Introduction}

Bishop and Phelps proved that the set of norm attaining functionals on a Banach space is norm dense in the topological dual  \cite{BP}.  Bollob\'{a}s  proved a \com{quantitative version} of the
Bishop-Phelps theorem \cite{Bol}. In order to state that result we  recall some notation. 
Given   a  Banach space $X$, we denote the unit sphere of $X$ by $S_X$ and the closed unit ball by
$B_X.$ The topological dual of $X$  is denoted by   $X^\ast.$   The following result can be found in  \cite[ Theorem 16.1]{BoDu}  and \cite[Corollary 2.4]{CKMMR} and it is  known as
the Bishop-Phelps-Bollob\'{a}s theorem: 

{\it Let $\varepsilon > 0$ be arbitrary. If $x \in B_X$ and $x^\ast \in S_{X^\ast}$ are such that
	$ \vert 1 - x^\ast(x ) \vert <  \dfrac{\varepsilon^2}{2},$  then there are elements $y\in S_X $ and
	$ y^\ast \in S_{X^\ast}$ such that $y^\ast (y)=1$,  $ \Vert y-x \Vert < \varepsilon$  and $
	\Vert y^\ast -x^\ast \Vert < \varepsilon $.}

In 2008  Acosta, Aron, García and Maestre  \cite{AAGM}	
	introduced the so-called Bishop-Phelps-Bollob{\'a}s property for operators (see \cite[Definition 1.1]{AAGM}) and provided some  pairs of Banach spaces with that property. Afterwards a number of interesting results on that topic appeared. The survey paper \cite{Acbc} contains many of such results.  Very recently the authors introduced the 
Bishop-Phelps-Bollob{\'a}s property for  positive operators between Banach lattices \cite{ASpos}.   In order to recall such property we need  some notions.
The concepts in the first definition  are standard and can be found, for instance, in \cite{AbAl}.

\begin{definition} 
	\label{def-lattice}
	An \textit{ordered vector space} is a real vector space $X$  equipped with a vector space order, that is,  an order relation $\le$  on $X$ that  is compatible with the algebraic structure of $X$.   An ordered vector space is called a  \textit{Riesz space}  if every pair of vectors has a least upper bound  and a greatest lower bound.   In a Riesz space $X,$   given two elements $x$ and $y$ in $X,$    we denote by $x \wedge y, $  $x \vee y, $  $\vert x\vert ,$ $x^+ ,$ and $x^-$  the infimum of $x$ and $y,$ the supremum of $x$ and $y,$  the supremum of $x$ and  $-x,$ , the supremum of $x$ and $0$, and the supremum of $-x$ and $0,$ respectively. 
	 A norm $\Vert \ \Vert $ on a Riesz space   $X$ is said  to be a   \textit{lattice norm} whenever $|x| \leq |y|$ implies $\Vert x\Vert  \le \Vert y \Vert  $. \textit{A normed  Riesz space} is a  Riesz space equipped  with a lattice norm. A normed Riesz space whose norm is complete  is called a \textit{Banach lattice}.
	 
	 In case that  $(\Omega,   \mu)$ is a 
	 measure space,  we denote by $L^0(\mu)$  the space of (equivalence classes of
	 $\mu$-a.e. equal) real valued measurable functions on $\Omega.$ We say that a Banach space $X$ is a \textit{Banach function space}
	 on $(\Omega,  \mu)$ if  $X$ is an ideal in $L^0 (\mu)$ and whenever $x,y\in
	 X$ and $|x| \le |y|$ a.e., then $\Vert x \Vert \le \Vert y \Vert .$

\end{definition}

For any Banach lattice $X,$ we will denote by $X^+$ the set of \textit{positive elements} in $X,$ that is, $X^+= \{ x \in X: 0 \le x\}.$
Recall that the  dual  of any normed Riesz space is a Banach lattice   (see \cite[Theorem 4.1]{AlBur}).  We will use  that in a Riesz space $X ,$  any   element   $x \in X$ satisfies  
$x= x^{+} - x^- ,$ $\vert x \vert = x^{+} +  x^-  $ and  $x^+ \wedge x^- = 0 $   (see  \cite[Theorem 1.5]{AlBur}).

If $X$ and $Y$ are Banach spaces, we denote by $L(X,Y)$ the space of all  bounded and linear operators from $X$ to $Y.$
A linear mapping  $T: X \llll Y$ between two ordered vector spaces is called \textit{positive} if $x \geq 0$ implies $T(x) \geq 0$.

Now we recall the notion of  Bishop-Phelps-Bollob{\'a}s property for  positive operators and introduce two new properties for functionals  that will be useful in this paper.

\begin{definition}[{\cite[Definition 1.3]{ASpos}}]
	\label{def-BPBp-pos}
	Let $X$  and $Y$ be    Banach lattices and $M $ a subspace of $L(X,Y).$  The subspace $M$  is said to have the {\it Bishop-Phelps-Bollob{\'a}s property for  positive operators}   if for every $  0 < \varepsilon  < 1 $  there exists $ 0< \eta (\varepsilon) < \varepsilon $ such that for every $S\in S_M$, such that $S \ge 0$,  if $x_0 \in S_X$ satisfies
	$ \Vert S (x_0) \Vert > 1 - \eta (\varepsilon)$, then
	there exist an element $u_0 \in S_X$  and a positive  operator $T \in S_M$ satisfying the following conditions
	$$
	\Vert T (u_0) \Vert =1, \sem \Vert u_0- x_0 \Vert < \varepsilon \seg \text{and}
	\sem \Vert T-S \Vert < \varepsilon.
	$$
\end{definition}

Acosta and Soleimani-Mourchehkhorti proved  that  the pair $(L_\infty (\mu), L_1 (\nu))$ has the  Bishop-Phelps-Bollob{\'a}s property for positive operators,  for any positive measures $\mu $ and $\nu$ (see \cite[Theorem 1.6]{ASpos}). The same property holds for the pair $(c_0, L_1 (\mu))$ for any positive  measure $\mu$ (see \cite[Theorem 1.7]{ASpos}).  	These results were extended in \cite{ASpos2} to  the pair $(c_0, Y)$ in case that $Y$ is a uniformly monotone Banach lattice and $(L_\infty (\mu), Y),$ whenever $Y$ is a uniformly monotone Banach lattice with a weak unit.

\begin{definition}
	\label{def-BPBpp}
	We say that a Banach lattice $X$   has the  {\it Bishop-Phelps-Bollob{\'a}s property for  positive functionals}  (we will  write  BPBp instead of Bishop-Phelps-Bollob{\'a}s property)  if for every $  0 < \varepsilon  < 1 $  there exists $ 0< \eta (\varepsilon) < \varepsilon $ such that for every $x^* \in S_{X^*}$, such that $x^* \ge 0$,  if $x_0 \in S_X$ satisfies
	$ x^* (x_0)  > 1 - \eta (\varepsilon)$, then
	there exist an element $y \in S_X$  and a positive  functional  $y^*\in S_{X^*}$ satisfying the following conditions
	$$
	y^*(y) =1, \sem \Vert y- x_0 \Vert < \varepsilon \seg \text{and}
	\sem \Vert y^*-x^* \Vert < \varepsilon.
	$$
	A Banach lattice $X$   has the  {\it strong Bishop-Phelps-Bollob{\'a}s property for  positive functionals}   if  it has the Bishop-Phelps-Bollob{\'a}s property for  positive functionals and additionally the element $y$ appearing in that definition is positive.
\end{definition}

Trivially the strong Bishop-Phelps-Bollob{\'a}s property for positive functionals implies the   Bishop-Phelps-Bollob{\'a}s property for positive functionals.

In Section 2 we obtain a general version of Bishop-Phelps-Bollob{\'a}s property for Banach lattices where the elements in  the Banach space and the functionals are  positive (see Proposition \ref{pro-BPBp-pp-general}). We also prove an intrinsic characterization of the strong BPBp for positive functionals (see  Theorem \ref{teo-UMOE-BPBpp}). In the case of the Bishop-Phelps-Bollob{\'a}s property for positive  functionals, we show that finite-dimensional Banach lattices always have such property (see Corollary  \ref{pro-finite-BPBpp}).  We also  provide a sufficient  and a necessary condition  for the BPBp for positive functionals
in Theorem \ref{teo-SM-HNA-BPBpp}.  

In Section 3  we  prove that the spaces $L_p (\mu),$ $C(K)$  and $\mathcal{M} (K)$  satisfy  the BPBp for positive functionals, for any positive measure $\mu$ and any compact and Hausdorff topological space $K$ (see Corollary \ref{co-Lp-C(K)-BPBpp}). We also  show that $\R^2,$ endowed with an absolute norm,  satisfies the assumptions  in Theorem \ref{teo-SM-HNA-BPBpp}. However, the case of the three dimensional Banach lattice is different (see Examples \ref{exam-dim-3-N-HNAp} and \ref{exam-dim-3-N-SM}). We also provide an example showing that not every Banach lattice has the Bishop-Phelps--Bollob{\'a}s property for positive functionals (see Example \ref{ell2-not-WM}).

We point out that throughout this paper we consider only real Banach spaces.

\vskip4mm

\section{The   results}

The goal of this section is to provide classes of Banach lattices 
with the strong Bishop-Phelps-Bollob{\'a}s property for positive functionals and the  Bishop-Phelps-Bollob{\'a}s property for positive functionals.

\begin{definition}
	\label{def-HNAp}
	A  real Banach lattice $X$ has \textit{the hereditary norm attaining property} (HNAp in short) if  it satisfies the following condition
	$$
	x^* \in X^*, x \in X, x^* (x) = \Vert x^* \Vert  \Vert x \Vert  \ \Rightarrow  \ 
	x^{* ^+} (x^+) = \Vert x^{* ^+} \Vert  \Vert x^+  \Vert ,~~~~ x^{*^-} (x^-) = \Vert x^{*^-} \Vert  \Vert x^- \Vert .
	$$
\end{definition} 

In the proof of the main results  the following simple facts will be useful.

\begin{proposition}
	\label{rem-state-dual-positive}
	Let $X$ be a Banach lattice.
\begin{enumerate}
	\item[i)] If $x \in S_X,$  $x^* \in S_{X^*}$ and $x^*(x)=1$ then $\vert x^* \vert ( \vert x \vert )=1,$ therefore
	$$
	x^{*+}(x^-)=0=  	x^{*-}(x^+)
	$$
		\item[ii)]   Assume that it is satisfied 
		$$
		x \in S_X, x^* \in S_{X^*}, x^* (x)=1 \ \Rightarrow \ \Vert x^{*+} \Vert \, \Vert x^+ \Vert  +  \Vert x^{*-} \Vert \, \Vert x^- \Vert \le 1.
		$$
		Then $X$ has the hereditary norm attaining property.
	\end{enumerate}	 
\end{proposition} 
\begin{proof}
For the proof of i) let us notice that
	\begin{align}
	1  &  =  x^*(x)   
	\nonumber \\
	&  =   \bigl( x^{*+} -  x^{*-} \bigr)    \bigl( x^{+} -  x^{-} \bigr)
	\nonumber  \\
	&  \le     x^{*+} (x^+) +   x^{*-} (x^-) 
	\nonumber   \\
	& \le     \bigl(  x^{*+}  +   x^{*-} \bigr )  \bigl( x^+ +   x^{*-} \bigr) 
	\nonumber  \\
	& =   \vert x^* \vert      \bigl( \vert x \vert \bigr)  
	\nonumber  \\
	& \le 1.
	\nonumber
	\end{align}
	Hence  $\vert x^* \vert ( \vert x \vert )=1$ and 
	$$
	x^{*+}(x^-)=0=  	x^{*-}(x^+).
	$$
	Now we prove ii). Let us notice that  it suffices to check that the condition in Definition \ref{def-HNAp} is satisfied for normalized elements. Let $x \in S_X $ and $x^*\in S_{X^*}$  be elements such that $x^* (x)=1.$  Then by using the assumption we have that
	\begin{align}
	1  &  =  x^*(x)   
	\nonumber \\
	&  =   \bigl( x^{*+} -  x^{*-} \bigr)    \bigl( x^{+} -  x^{-} \bigr)
	\nonumber  \\
	&  \le     x^{*+} (x^+) +   x^{*-} (x^-) 
	\nonumber   \\
	& \le     \Vert   x^{*+}  \Vert  \  \Vert  x^+ \Vert    +  \Vert   x^{*-}  \Vert  \  \Vert  x^- \Vert 
	\nonumber  \\
	& \le 1.
	\nonumber
	\end{align}
	Hence  $    x^{*+} (x^+) =     \Vert   x^{*+}  \Vert  \  \Vert  x^+ \Vert $ and $ x^{*-} (x^-) =     \Vert   x^{*-}  \Vert  \  \Vert  x^- \Vert .$  So $X$ has the hereditary norm attaining property.
\end{proof}

From Bishop-Phelps-Bollob{\'a}s Theorem we will obtain the following general result. 

\begin{proposition}
	\label{pro-BPBp-pp-general}
	Let $X$ be a Banach lattice and $0  < \epsi < 1.$ If $x \in S_X$ is positive,  $x^* \in S_ {X^*}$ is positive and $x^* (x) >  1 - \frac{\epsi ^2}{2} $ there are  elements $y\in X^+ \cap S_X$ and $y^* \in X^{*+} \cap S_{X^*}$ satisfying the following conditions
	$$
	y^* (y)=1, \sem \Vert y - x \Vert <\epsi \sem \text{\rm and} \sem \Vert y^* - x^* \Vert < \epsi.
	$$
	\end{proposition}
\begin{proof}
By applying  the Bishop-Phelps-Bollob{\'a}s Theorem there are elements $z \in S_X$ and $z^* \in S_{X^*}$ such that
\begin{equation}
\label{dem-pro-BPBp-pos-pos}
z^*(z)=1, \sem \Vert z - x \Vert < \epsi \sem \text{and} \sem 
\Vert z^* - x^* \Vert < \epsi .
\end{equation}
Hence  $\vert z \vert  \in S_{X}$ and  $\vert z^* \vert  \in S_{X^*}.$ In view of Proposition 	\ref{rem-state-dual-positive}  we have that $\vert z^* \vert (\vert z \vert )=1. $
From \eqref{dem-pro-BPBp-pos-pos}, since $x$ is positive and $x^*$ is also positive we obtain that
\begin{equation}
\label{dem-pro-BPBp-pos-dist-absol}
 \Vert\vert  z  \vert - x \Vert < \epsi \sem  \text{and} \sem 
\Vert \vert z^* \vert  - x^* \Vert < \epsi .
\end{equation}
Hence the elements $y= \vert z \vert $ and $y^*= \vert z^* \vert$ satisfy the stated assertion.
\end{proof}

Now we  collect  a few known facts in the following statement.  They can be found in Theorems 1.5 and 1.7 and Exercise 2 in \cite{AlBur}.

\begin{lemma}
	\label{le-facts-lattices}
	Let $X$ be a  Banach lattice. The following statements hold.
	\begin{enumerate}
		\item[a)] If  $y, z \in X^+, y \wedge z=0, x= y-z \ \Rightarrow \ y=x^+, z = x^-.$
		\item[b)] $ \vert x-y \vert = x \vee y - x \wedge y, \sep \forall x,y \in X.$
		\item[c)]
		$$
		x \wedge y  = \frac{1}{2} \bigl(     x + y  - \vert x-y \vert \bigr)  , \sem  x \vee y  = \frac{1}{2} \bigl(     x + y + \vert x-y \vert \bigr)  , \seg \forall x,y \in X.
		$$
		\item[d)]
		$$
		\vert x \vert \wedge \vert y \vert = \frac{1}{2} \Bigl \vert   \vert x + y \vert - \vert x-y \vert \Bigr \vert , \sem \vert x \vert \vee \vert y \vert = \frac{1}{2} \Bigl \vert   \vert x + y \vert + \vert x-y \vert \Bigr \vert , \seg \forall x,y \in X.
		$$
		
		\item[e)]
		$$
		x, y \in X, x \bot y \ \Rightarrow \ sx \bot ty, \seg \forall s,t \in \R.
		$$
		\item[f)]
		$$
		x, y \in X^+,  x  \wedge y =0  \ \Rightarrow \ \vert sx+t y \vert = \vert s \vert x + \vert  t \vert y, \seg \forall s,t \in \R.
		$$
		As a consequence
		$$
		\Vert sx + t y \Vert \ge \max \bigl\{  \vert s \vert \, \Vert x \Vert,   \vert t \vert \, \Vert y\Vert\bigr\}.
		$$
	\end{enumerate}
\end{lemma}

Later we will  also use the following  result.

\begin{proposition}
	\label{pro-sep-pos-ort}
	Let $X$be a Banach lattice and assume that $x, y \in X^+$  satisfy  $x \ne 0$ and $x  \wedge y =0.$ Then 
	there is a positive functional $x^* \in S_{X^*}$ such that  $x^*(x)= \Vert x \Vert $ and $x^*(y)=0.$
\end{proposition}
\begin{proof}
		In case that  $y=0$ the result follows from  Hahn-Banach extension theorem and part i) of
		Proposition \ref{rem-state-dual-positive}.
	
	If $y\ne 0$  from Lemma  \ref{le-facts-lattices} we obtain that $x$ and $y$ are linearly independent since
	\begin{equation}
	\label{x-y-max}
	\Vert sx + t y \Vert \ge \max \bigl\{  \vert s \vert \, \Vert x \Vert,   \vert t \vert \, \Vert y\Vert\bigr\}, \seg \forall s,t \in \R.
	\end{equation}
	
	We define the functional $y^*$ on $\spa\{x,y\}$  by 
	$$
	y^* (sx+ty) = s \Vert x \Vert, \seg  s,t \in \R.
	$$
	From \eqref{x-y-max} we have that $y^*$ is a continuous linear functional and satisfies
	$$
	y^* (x)= \Vert  x \Vert,  \sep \Vert y^* \Vert =1 \sep \text{and} \sep y^*(y)=0.
	$$
	We claim that $y^*$ is  a positive functional. Otherwise there would be real numbers $a<  0$ and $b> 0$ such that $ax+by \ge 0 $ and so $y \ge \frac{-a}{b} x > 0.$
	Therefore
	$$
	0= x \wedge y \ge \min \Bigl\{ 1, \frac{-a}{b} \Bigr\} x > 0,
	$$
	which is a contradiction. So $y^*$ is a positive functional on $\spa\{x,y\}$. 
	
	Finally we check that $\spa\{x,y\}$ is a Riesz subspace of $X.$ From c) and f) in  Lemma  \ref{le-facts-lattices} we have that
	\begin{align}
	(ax+by) \wedge (cx+dy)   &  =    \frac{1}{2} \bigl(  ax+by + cx+dy -  \vert ax+by -cx-dy \vert \bigr)
	\nonumber \\
	&  = \frac{1}{2} \bigl(  ( a +c)  x +  ( b + d)  y -\vert  a - c  \vert x -\vert  b-d  \vert y \bigr) .
	\nonumber  
	\end{align}
	and
	\begin{align}
	(ax+by) \vee (cx+dy)   &  =      \frac{1}{2} \bigl(  ax+by + cx+dy + \vert ax+by -cx-dy \vert \bigr)
	\nonumber \\
	&  = \frac{1}{2} \bigl(  ( a +c)  x +  ( b + d)  y +\vert  a - c  \vert x +\vert  b-d  \vert y \bigr) .
	\nonumber    
	\end{align}
	Hence $\spa\{x,y\}$ is stable under suprema and infima. From \cite[Theorem 39.2]{Zaan}
	there is a positive functional $x^*\in X^*$ such that 
	$$
	\Vert x^* \Vert = \Vert y^* \Vert = 1 \sem \text{and} \sem x^* _{\vert \spa\{x,y\}} = y^*
	$$
	and the proof is finished.
\end{proof}

Firstly we recall the notion of uniformly monotone Banach lattice, which  is well known,  and introduce a new  geometric property  for Banach lattices.

\begin{definition}
	\label{def-UM}
	A real Banach lattice $X$ is \textit{uniformly monotone}  (UM), if for every $0<\varepsilon<1,$ there is 
	$0  <  \delta < \varepsilon$ satisfying the following property
	$$
	x , y  \in X^+   ,   \Vert x+ y\Vert  \leq 1  , \sep \Vert x \Vert >  1 -\delta\ \Rightarrow\ \Vert y\Vert < \varepsilon .
	$$
\end{definition}

\begin{definition}
	\label{def-UMOE}
	A real Banach lattice $X$ is \textit{uniformly  monotone  for orthogonal elements} (UMOE in short), if   for every $0<\varepsilon<1$ there is 
	$0  <  \delta < \varepsilon$ such that
	$$
	x \in B_X  , \Vert x^+ \Vert >  1 -\delta \ \Rightarrow \  \Vert x^-\Vert < \varepsilon .
	$$
\end{definition}

It is clear that  any uniformly monotone Banach lattice is  uniformly monotone for orthogonal elements. It is shown in \cite[Theorem 6]{HKM} that for Banach function spaces  uniform monotonicity and uniform monotonicity for orthogonal elements  are equivalent properties.

\begin{remark}
	A  real Banach lattice $X$ is uniformly  monotone  for orthogonal elements if and only if for every $0<\varepsilon<1,$ there is 
	$0  <  \delta < \varepsilon$ such that
	$$
	x , y  \in X^+   ,   x\wedge y = 0,  \Vert x+ y\Vert  \leq 1     , \Vert x \Vert >  1 -\delta\ \Rightarrow\ \Vert y\Vert < \varepsilon .
	$$		
\end{remark}

Now we can prove the following characterization.

\begin{theorem}
	\label{teo-UMOE-BPBpp}
	Let $X$ be a Banach lattice. Then $X$ is   uniformly monotone  for orthogonal elements if and only if it has the strong Bishop-Phelps-Bollob{\'a}s property for positive functionals. Moreover the function $\eta$ in Definition \ref{def-BPBpp} depends on the function $\delta $ satisfying the condition of uniformly monotonicity for orthogonal elements (see Definition \ref{def-UMOE}).
		\end{theorem}
\begin{proof}
	Assume that $X$ is   uniformly monotone  for orthogonal elements with the function $\delta .$ 
	Let be $0  < \epsi < 1,$ $x \in S_X,$ $x^* \in S_{X^*}$ such that 
	$$
	x^* \ge 0 \sem \text{and} \sem x^* (x) > 1 - \min \Bigl\{ \frac{\epsi^2}{2}, \delta (\epsi)\Bigr\}.
	$$
	Since $x^* \ge 0$ we also have   $x^* (\vert x \vert ) \ge x^* (x) > 1 -  \frac{\epsi^2}{2}.$ 
	By Bishop-Phelps-Bollob{\'a}s theorem there are $y\in S_X$ and $y^* \in S_{X^*}$ satisfying 
	\begin{equation}
	\label{abs-v-y*-y}
	y^*(y)=1, \sep \Vert y -  x  \, \Vert < \epsi \sep \text{and} \sep \Vert y ^*  -  x^*  \Vert < \epsi.   
	\end{equation}
	From the inequality $ x^*(x) = x^* (x^+) - x^* (x^-) >  1 - \delta (\epsi)$, since $x^* \ge 0$ we get that  $x^* (x^+) > 1 - \delta (\eee)$ and so $\Vert x^+ \Vert  > 1 - \delta (\epsi).$ By using that $X$ is  uniformly monotone  for orthogonal elements 
	we get that $ \Vert x^- \Vert < \epsi $ and so
	$$
	\Vert x - \vert x \vert \, \Vert = 2 \Vert x^- \Vert < 2 \epsi.
	$$
	From \eqref{abs-v-y*-y}, 
	we obtain that 
	$$
	\Vert \, \vert y \vert - x \Vert \le  \Vert \, \vert y \vert - \vert  x  \vert \Vert  + \Vert \, \vert x \vert - x \Vert   \le \Vert y - x \Vert + 2 \Vert x^- \Vert < 3 \epsi.
	$$
	In view of the previous inequality  and \eqref{abs-v-y*-y} we proved that $X$ has the strong Bishop-Phelps-Bollob{\'a}s property for positive functionals since $\vert y^\ast \vert $ attains its norm at $\vert y \vert $ and $\Vert \vert y^* \vert - x^* \Vert \le \Vert y^* - x^* \Vert < \epsi.$  Let us also notice that the function $\eta $  in Definition \ref{def-BPBpp} depends only on the function $\delta $ appearing in the definition of uniformly monotonicity for orthogonal elements.
	
	Assume now that $X$ has the strong Bishop-Phelps-Bollob{\'a}s property for positive functionals
	with the function $\eta.$ Let $0 < \eee < 1$  and $x \in B_X$ such that  $\Vert x^+ \Vert   > 1 - \eta (\eee).$
It is known that   $ x^+ \wedge x^-=0 .$ So
	by  Proposition \ref{pro-sep-pos-ort} there is a positive functional $x^* \in S_{X^*}$ such that $x^*(x)= \Vert x^+ \Vert > 1 - \eta (\eee) .$ Therefore $x^*(\frac{x}{\Vert x \Vert}) > 1 - \eta (\eee) .$
	
	By the assumption there are elements $z \in S_X$ and $z^* \in S_{X^*}$ such that
	\begin{equation}
	\label{z*-z-bpt}
	z^*\ge 0, \sep z \ge 0, \sep z^* (z) =1,  \sep \Bigl \Vert z - \frac{x}{\Vert x \Vert} \Bigr \Vert< \eee \sep  \text{and }  \Vert z^*- x^* \Vert < \eee.
	\end{equation}
 By  \eqref{z*-z-bpt} we conclude that 
	$$
\Bigl \Vert	  \frac{x^-}{\Vert x \Vert}   \Bigr \Vert= \Bigl \Vert	 z^- -  \frac{x^-}{\Vert x \Vert} \Bigr \Vert \le \Bigl \Vert z - \frac{x}{\Vert x \Vert} \Bigr \Vert < \eee.
	$$
As we proved that $ \Vert x^-   \Vert < \eee \Vert x\Vert \leq\eee,$ $X$ is uniformly monotone for orthogonal elements.
\end{proof}

Our goal now is to provide classes of Banach spaces satisfying the Bishop-Phelps-Bollob{\'a}s property for positive functionals.

\begin{lemma}
	\label{le-pos-closed}
	For any Banach lattice  $X,$ the sets
$X^+$ and $X^-= \{ x \in X: x \le 0\}$ are closed.	
\end{lemma}
\begin{proof}
	It is satisfied that
	$$
	\Vert x^+- y^+ \Vert \le \Vert x - y \Vert, \seg \forall  x \in X,
	$$
	so the mapping $x \mapsto x^+$ is continuous on $X.$ As a consequence, $X^+= \{ x \in X: 0 \le x\}$ is  closed. The same argument  holds true for $X^-.$
	\end{proof}

\begin{proposition}
	\label{finite-newerer}
	Let $X$ and $Y$ be finite-dimensional Banach lattices. For every $\varepsilon > 0, $   there exists
	$\delta  > 0 $ such that whenever $S \in  S_{L(X,Y)}$ is a positive operator there is a  positive and linear operator $T \in  S_{L(X,Y)} $ such that the
	following conditions hold:

	\begin{enumerate}
		\item[i)] $ \Vert T- S \Vert < \varepsilon$, and
		
		\item[ii)]  for all $x_0 \in  S_X $ satisfying $\Vert S(x_0) \Vert > 1-\delta ,$ there is $u_0 \in   S_X$ such that $\Vert T(u_0)\Vert=1$ and such
		that $ \Vert u_0- x_0 \Vert < \varepsilon .$
	\end{enumerate}
	
\end{proposition}

\begin{proof}
	The proof is similar to the proof of [Proposition 2.4] in \cite{AAGM}. The only difference is that  in the current proof  we assume that all operators are positive. Notice that because of Lemma \ref{le-pos-closed}, the set of positive (and bounded) operators between two Banach lattices is  a closed subset of  the space of bounded and linear operators between them.
\end{proof}

So from this observation we conclude the next result.

\begin{corollary}\label{finite-new}
	Assume that $X$ and $Y$ are finite-dimensional Banach lattices. Then $L(X,Y)$ has the  Bishop-Phelps-Bollob\'{a}s property for positive operators.	
\end{corollary}

\begin{corollary}\label{pro-finite-BPBpp}
	Every finite dimensional Banach lattice has  Bishop-Phelps-Bollob\'{a}s property for positive functionals.
\end{corollary}

In order to state another  result for  the Bishop-Phelps-Bollob{\'a}s property for positive functionals, we introduce the following two notions.

\begin{definition}
	\label{def-SM}
	A real Banach lattice $X$ is \textit{strongly  monotone} (SM for short)  if for every $0<\varepsilon<1$ there is 
	$0  <  \delta < \varepsilon$ satisfying the following property
	$$
	x \in B_X, \Vert x^+ \Vert   > 1 -\delta \ \Rightarrow \  \exists b \in [0 , 1]  : \frac{x^+}{\Vert x^+\Vert} - b x^- \in S_X \sep \text{and} \sep  \Vert bx^- - x^-\Vert < \varepsilon.
	$$
\end{definition}

\begin{definition}
	\label{def-WM}
	A  real Banach lattice $X$ is \textit{weakly  monotone}  (WM for short) if for every $0<\varepsilon<1$ there is 
	$0  <  \delta < \varepsilon$ such that
	$$
	x \in B_X,  \Vert x^+ \Vert  > 1 -\delta \ \Rightarrow  \ \exists y \in S_X :  y^+ \in S_X \sep  \text{ and} \sep  \Vert y- x\Vert < \varepsilon.
	$$	
\end{definition}

It is clear that  any  Banach lattice that is uniformly monotone for orthogonal elements is also strongly  monotone. Strongly monotonicity implies weakly monotonicity.

\begin{theorem}
	\label{teo-SM-HNA-BPBpp}
	Let $X$ be a Banach lattice. If $X$ is  strongly  monotone and has the hereditary norm attaining property then $X$ has the Bishop-Phelps-Bollob{\'a}s property for positive functionals.   The  Banach lattice $X$ is    weakly monotone  whenever it  has the Bishop-Phelps-Bollob{\'a}s property for positive functionals.
\end{theorem}
\begin{proof}
We begin by proving the first statement.	Let be $0 < \eee < 1$ and $\delta $  the function satisfying Definition \ref{def-SM}. We can assume that $\delta (t) \le t$ for every $t \in (0,1).$ 
		Assume that $x^* \in S_{X^*}$ is positive and $x \in S_X$ are such that $x^*(x) > 1 - \frac{\delta ^2 (\eee)}{2}.$ By applying Bishop-Phelps-Bollob{\'a}s theorem there are $y^* \in S_{X^*} $ and $y \in S_X$ satisfying also 
		\begin{equation}
		\label{y*-y-bpbt}
		y^*(y)=1, \sep \Vert y^* - x^* \Vert < \delta (\eee)  \sep \text{and} \sep \Vert y - x \Vert < \delta (\eee),
		\end{equation}
		so
		\begin{equation}
		\label{y*+-x*+}
	\Vert y^{*+} - x^{*} \Vert = 	\Vert y^{*+} - x^{*+} \Vert  < \delta (\eee).
		\end{equation}	
		Therefore 
		$$
		\Vert y^{*+} \Vert	 > 1 - \delta (\eee) \ge  1 - \eee > 0
		$$
		and 
			\begin{equation}
		\label{y*-x*-}
		\Vert y^{*-}  \Vert = \Vert y^{*-} - x^{*-} \Vert  \le 
		\Vert y^* - x^* \Vert <  \delta (\eee).
		\end{equation}		
		Hence
		\begin{equation}
		\label{y*+-y*}
		\Bigl \Vert  \frac{y^{*+}}{\Vert y^{*+}\Vert}  - y^{*} \Bigr\Vert \le  \bigl \vert 1 - \Vert  y^{*+} \Vert \, \bigr \vert  +  \Vert  y^{*-} \Vert <  2 \eee.
		\end{equation}	
		By Proposition \ref{rem-state-dual-positive} we know that $y^{*+}(y^-)=y^{*-}(y^+)=0.$ As a consequence we have that
		$$
		1= y^*(y)= y^{*+}(y^+) + y^{*-}(y^-).
		$$
		So  we obtain that
		\begin{equation}
		\label{norm-y+}
		\min\Bigl\{ \Vert y ^+ \Vert, \Vert y^{*+} \Vert \Bigr\}  \ge  y^{*+} (y^+) =  1- y^{*-} (y^-)  \ge    1 - \Vert  y^{*-} \Vert  > 1- \delta (\eee) \ge 1-  \eee > 0.
		\end{equation}	
		Since $X$ is strongly  monotone there is $b \in [0,1]$ such that 
		\begin{equation}
		\label{z-norm-1}
		\Bigl \Vert \frac{ y^+}{ \Vert y^+ \Vert} - by^- \Bigr \Vert =1 \sem \text{and} \sem \Vert by^- - y ^- \Vert < \eee,
		\end{equation}
		so
		\begin{align}
		\Bigl \Vert  \frac{ y^+}{ \Vert y^+ \Vert} - by^-  -y \Bigr \Vert   &  \le    
		\Bigl \Vert  \frac{ y^+}{ \Vert y^+ \Vert} - y^+ \Bigr \Vert +
		\Vert  y^- - by^-  \Vert
		\nonumber \\
		&  < 1 - \Vert y^+ \Vert + \eee 
		\nonumber  \\
		&  < 2 \eee  \ \ \ \text{\rm (by \eqref{norm-y+})}.
		\nonumber 
		\end{align}
		If we put $z= \frac{y^+}{\Vert y^+ \Vert} - by^-$we have that $z \in S_X$ in view of \eqref{z-norm-1}   and  we just checked that $\Vert z -y \Vert < 2 \eee.$
		By using also  \eqref{y*-y-bpbt} we obtain that 
		\begin{equation}
		\label{z-x}
		\Vert z-x \Vert  \le \Vert z-y \Vert + \Vert y-x \Vert  < 3        \eee .
		\end{equation}	
		We also have that
		\begin{align}
		\label{y^*+-norm-x*}
		\Bigl \Vert  \frac{ y^{*+}}{ \Vert y^{*+} \Vert} - x^* \Bigr \Vert    &  \le    
		\Bigl \Vert  \frac{ y^{*+}}{ \Vert y^{*+} \Vert} - y^* \Bigr \Vert + \Vert y^* - x^* \Vert
		\nonumber \\
		&  < 3 \eee \sem \text{(by   \eqref{y*+-y*}  and \eqref{y*-y-bpbt})}.
		\end{align}
		Since $y^{*+}(y^-)=0$ it is also immediate that 
		$$
		\frac{ y^{*+}}{\Vert    y^{*+} \Vert } (z) = \frac{ y^{*+} (y^+)}{\Vert    y^{*+} \Vert \, \Vert y^+ \Vert  }.
		$$
		By using that  $X$  has the hereditary norm attaining property we have that $y^{*+}(y^+) = \Vert    y^{*+} \Vert \, \Vert y^+ \Vert ,$ so
		$\frac{ y^{*+}}{\Vert    y^{*+} \Vert } (z)=1.$ 
		
		By taking into account also  \eqref{z-x} and  \eqref{y^*+-norm-x*} we proved that $X$ has the Bishop-Phelps-Bollob{\'a}s property for positive functionals.

	Now we prove the second assertion. Let assume that $X$  has  BPBp for positive functionals with the function $\eta $ and  $0 <\eta  (\epsi)< \varepsilon$ for every positive real number $\epsi.$

		Let  fix   $0<\varepsilon<1,$  $0<\eta= \eta (\frac{\epsi}{2})  <\frac{\epsi}{2} $  and $x \in B_X$  such that $\Vert x^+ \Vert   >  1 -\eta > 1- \frac{\varepsilon}{2} > 0 .$ From  Proposition \ref{pro-sep-pos-ort}, there is a positive functional $x^* \in S_{X ^*}$  such that
		$$
		x^*  (x^+) = \Vert x^+ \Vert  >   1-\eta  \sem \text{and} \sem x^*(x^-)=0. 
		$$
		Hence $x^* \bigl( \frac{ x}{ \Vert x \Vert }\bigr )= x^* \bigl( \frac{ x^+}{ \Vert x \Vert }\bigr ) > 1- \eta.$ 
Since $X$ has  the BPBp for positive functionals  with the function $\eta ,$  there are  a positive functional $y^* \in S_{X^*} $    and $y \in S_X$ satisfying also 
\begin{equation}
\label{BpBpp-proof}
y^* (y) = 1, \sem \text{and} \sem    \Bigl\Vert y -\frac{x}{\Vert x\Vert }  \Bigr\Vert < \frac{\epsi}{2} .
\end{equation}
Since $y^*$ is a positive functional we  obtain that $y^* (y^+)=1$ and so $\Vert y^+ \Vert =1.$    From \eqref{BpBpp-proof} we deduce that  
$$
 \Vert y - x \Vert  \le  \Bigl\Vert y -\frac{x}{\Vert x\Vert }  \Bigr\Vert  + \Bigl\Vert \frac{x}{\Vert x\Vert }  - x  \Bigr\Vert < \frac{\epsi}{2}   + 1 - \Vert x \Vert  \le  \frac{\epsi}{2}   + 1 - \Vert x^+ \Vert   < \epsi.
  $$
 Therefore $X$ is   weakly  monotone.
\end{proof}

Now we make a sketch  of some  results known until now

$$
\text {UM}  \ \Rightarrow  \  \text    {UMOE} \ \Rightarrow  \      \text    {SM}  \ \Rightarrow  \ 
    \text    {WM}   
$$

\vskip3mm

$$
 \text {On Banach function spaces} \hspace*{1cm} ~~~~~~~  \text {UM} \ \Leftrightarrow  \ \text{UMOE}
$$

\vskip3mm

$$
\text{strong BPBp for  positive functionals} \ \Rightarrow  \ \text{BPBp for positive functionals}
$$

\vskip3mm

$$
\text{strong BPBp for  positive functionals} \ \Leftrightarrow  \ \text{UMOE}
$$

\vskip3mm

$$
 \text{finite dimensional} \ \Rightarrow  \    \text{BPBp for positive functionals}
$$

\vskip3mm

$$
\text{SM + HNAp} \ \Rightarrow  \ \text{BPBp for positive functionals} \ \Rightarrow  \   \text{WM} 
$$

\section{Examples}

Our  goal now is to provide examples of spaces satisfying the  sufficient condition in  Theorem \ref{teo-SM-HNA-BPBpp}. Later we will  exhibit more examples showing that the converse implications that we wrote above  does not hold. We also show with these examples that the Bishop-Phelps-Bollob{\'a}s property for positive operators is not trivially satisfied for any Banach lattice.

\begin{proposition}
	\label{pro-Lp-CK-VWUM-HNAp}
	The following Banach lattices  are strongly  monotone and have the hereditary norm attaining property:
	\begin{enumerate}
		\item[i)]  $L_p (\mu)$ for any positive measure $\mu$ and $1 \le p < \infty .$
		\item[ii)] $ C(K)$ for any compact and Hausdorff topological space $K.$
		\item[iii)]  $\mathcal{M} (K)$, the space of regular Borel measures on a compact  and Hausdorff topological space $K.$ 
	\end{enumerate}
In fact,  $L_p (\mu)$ for $1 \le  p  < \infty$   and  $\mathcal{M} (K)$ are uniformly monotone.
\end{proposition}
\begin{proof}
	\textbf{i)}  For $1 <  p < \infty$ the space $L_p(\mu)$ is  uniformly convex and so uniformly  monotone.
	
	For $p=1$ we check that $L_1(\mu)$ is also uniformly monotone. If  $ \eee >0,$ $f, g \in L_1 (\mu)$ are positive elements  such that
	$$
	\Vert f \Vert _1 + \Vert g \Vert _1 = 	\Vert f+g \Vert _1 \le 1  \sep \text{and} \sep   \Vert f \Vert _1 >  1 - \eee \ \Rightarrow  \ \Vert g \Vert _1 <  \eee.
	$$
	Now we show that $L_p(\mu)$ has the hereditary norm attaining property. 
	If $p > 1,$  since the dual of $L_p(\mu)$ is identified  with $L_q (\mu)$  where $\dfrac{1}{p} +  \dfrac{1}{q}=1, $  let  $f \in S_{L_p (\mu)}$ and $g \in S_{L_q (\mu)}.$ 
	Since $f ^+ $ and $f^-$ have disjoint support we have that
	$$
	1 = \Vert f \Vert _p ^p=\Vert f^+\Vert_p ^p + \Vert f^- \Vert_p ^p .
	$$
	From H\"{o}lder inequality it follows that
	$$
	\Vert g^+ \Vert _q  \ \Vert f ^+ \Vert _p +  \Vert g ^- \Vert _q  \ \Vert  f ^- \Vert _p \le 1.
	$$
	The previous inequality implies the hereditary norm attaining property in view of Proposition \ref{rem-state-dual-positive}.
		
	In case that $p=1, $ if $x^* \in S_{ L_1(\mu)^*}$ and $f \in S_{L_1(\mu)} $ we have that
	$$
	\Vert x ^{*+} \Vert \   \Vert f ^+ \Vert _1 + 	\Vert x ^{*-} \Vert \   \Vert f ^- \Vert _1   \le   \Vert f ^+ \Vert _1 +  \Vert f ^- \Vert _1 =1,
	$$
	so  by using again Proposition \ref{rem-state-dual-positive}, $L_1(\mu) $ also has the hereditary norm attaining property.

	\textbf{ii)}  It is clear that 
	$$
	\Vert f \Vert = \max \bigl\{ \Vert f ^+\Vert , \Vert f^- \Vert\bigr\}, \sem \forall f \in C(K).
	$$
	So, if $0 < \eee < 1,$ $f \in B_{C(K)}$ and $\Vert f^+ \Vert > 1- \eee> 0,$ the element 
	$ \dfrac{ f^+}{\Vert f^+ \Vert } -f^- \in S_{ C(K)},$ so $C(K)$ satisfies Definition \ref{def-SM} with $\delta (\eee)= \eee$ and $b=1.$
	
	Now we check that $C(K)$ has the hereditary norm attaining property.  We use that the topological  dual of $C(K)$ is identified with $\mathcal{M}(K),$ the space of regular and Borel measures on $K$,  and   for the dual norm we have that
	\begin{equation}
	\label{M-L1-norm}		
	\Vert \mu \Vert =  		\Vert \mu^+ \Vert + 		\Vert \mu^- \Vert, \seg \forall \mu \in \mathcal{M} (K).
	\end{equation}
	As a consequence we obtain  that 
	\begin{equation}
	\label{proof-CK-HNAp}	
	\Vert x^{*+} \Vert \, \Vert x ^+ \Vert  + \Vert x^{*-} \Vert \, \Vert x ^- \Vert \le 
	\Vert x^{*+} \Vert  + \Vert x^{*-} \Vert  =1, \seg \forall x \in S_{C(K)}, x^* \in S_{ C(K) ^*} .
	\end{equation}
	Hence $C(K)$ has the hereditary norm attaining property.

	\textbf{iii)} From \eqref{M-L1-norm} it follows that $\mathcal{M}(K)$ is uniformly monotone, so it is strongly  monotone.  By using again that $\mathcal{M}(K)$ is an $L$-space and an argument similar to 
	\eqref{proof-CK-HNAp} it is immediate to obtain that $\mathcal{M}(K)$  also has the hereditary norm attaining property.
\end{proof}

\begin{corollary}
	\label{co-Lp-C(K)-BPBpp}
	For any positive measure $\mu$ and $1 \le p < \infty$ the space $L_p(\mu) $ has the  strong  Bishop-Phelps-Bollob{\'a}s property for positive functionals. The same property holds for $\mathcal{M}(K)$  for any compact and Hausdorff topological space $K.$  The space  $ C(K)$ has the   Bishop-Phelps-Bollob{\'a}s property for positive functionals.
\end{corollary}
\begin{proof}
	It suffices to use Proposition \ref{pro-Lp-CK-VWUM-HNAp}, Theorems \ref{teo-UMOE-BPBpp} and \ref{teo-SM-HNA-BPBpp}.
\end{proof}

	\begin{definition}
	A norm $\vert ~.~\vert$ on  $ \R^N$  is called \textit{absolute} if it satisfies that
	$$
	\vert  (x_i)  \vert =   \vert ( \vert x_i \vert ) \vert , \sem \forall  (x_i) \in \R^N .
	$$
	An absolute norm $\vert ~.~\vert$ is said to be \textit{normalized}   if $\vert e_i \vert =1$ for every $1\le i \le N $, where $\{ e_i: 1 \le i \le N \}$ is the canonical basis of $\R ^N$. 
\end{definition}

Next we show  that the assumptions in Theorem \ref{teo-SM-HNA-BPBpp}  implying the Bishop-Phelps-Bollobás Theorem for functionals are satisfied by $\R ^2$  with an absolute norm. Later we will exhibit examples showing that this is not the case for  $\R ^3.$  

\begin{proposition}
	\label{pro-R2-HNAp-UMOE}
	Let $\vert \ \vert $ be an absolute norm on $\R ^2.$ Then $X=(\R^2, \vert \ \vert )$ has the hereditary norm attaining property. The space $X$ is also strongly monotone.
\end{proposition}
\begin{proof}
	Firstly we prove that $X$ has the hereditary  norm attaining property. 
	
	Let $x=(a,b) \in S_X$ and $x^*= (u,v) \in S_{X^*}$ such  that  $x^*(x)=1.$ In  case that $x \ge 0, x \le 0, x^* \ge 0 $ or $  x^* \le 0$ the condition
	$$
	\Vert x^{*+} \Vert \, \Vert x^+ \Vert +\Vert x^{*-} \Vert \, \Vert x^- \Vert  \le 1
	$$
	is trivially satisfied.
	
	Otherwise  we have that
	\begin{equation}
	\label{signs-ab-uv}
	ab < 0 \sem \text{and} \sem uv < 0.
	\end{equation}
	Since 
	\begin{equation}
	\label{ineq-x*-x}
	1 = x^* (x)= ua + vb  \le   \vert ua  \vert  + \vert vb\vert  \le 1,
	\end{equation}
	we obtain that
	\begin{equation}
	\label{signs-ua-vb}
	ua >  0 \sem \text{and} \sem  vb > 0.
	\end{equation}
	
	In case that $a > 0,$ from \eqref{signs-ab-uv}  and \eqref{signs-ua-vb} we have that $b<0, u> 0$ and $v < 0,$ so
	$$
	x^+= (a,0), \sep x^-= (0,-b), \sep  x^{*+}= (u,0) \sep \text{and}  \sep  x^{*-}= (0, -v).
	$$  
	In view of \eqref{ineq-x*-x} we get that
	$$
	\Vert x^{*+} \Vert \, \Vert x^+ \Vert +\Vert x^{*-} \Vert \, \Vert x^- \Vert = \vert u \vert \,   \vert a \vert   + \vert v \vert \,  \vert b \vert \,  \le 1.
	$$
	
	If $a < 0,$ from \eqref{signs-ab-uv} we have that $b>0,  v > 0 $ and $u< 0,$  hence
	$$
	x^+= (0,b), \sep x^-= (-a,0), \sep  x^{*+}= (0,v) \sep \text{and}  \sep  x^{*-}= (-u,0).
	$$  
	By using again  \eqref{ineq-x*-x} we obtain that
	$$
	\Vert x^{*+} \Vert \, \Vert x^+ \Vert +\Vert x^{*-} \Vert \, \Vert x^- \Vert = \vert v \vert \,   \vert b \vert   + \vert u \vert \,  \vert a \vert \,  \le 1.
	$$
	By  Proposition  \ref{rem-state-dual-positive} we proved that $X$ has the hereditary norm attaining property.
	
	From \cite[Lemma 2.5]{AMS} it follows that $X$ is also strongly monotone.
\end{proof}

\begin{example}
	\label{exam-dim-3-N-HNAp}
	There is a norm on $\R^3$ which is absolute and normalized, but it does not satisfy the hereditary norm attaining property.
\end{example}

\begin{proof}
	Let 	$\vert \ \vert$   be the norm  on $\R^2 $ given by
	$$
	\vert (r,s) \vert = \max \Bigl\{ \vert r \vert, \vert s \vert, \frac{2}{3} \bigl( \vert r \vert + \vert s \vert \bigr)	\Bigr\}.
	$$
	Clearly the previous norm is a normalized  absolute norm on $\R^2$ and satisfies
	$$
	1= \Bigl \vert  \Bigl(\frac{\eee _1}{2},\eee_2 \Bigr) \Bigr \vert =\Bigl \vert  \Bigl(\eee_1, \frac{\eee _2}{2} \Bigr) \Bigr \vert = \Bigl \vert  \Bigl(\frac{3}{4}, \frac{-3}{4}\Bigr) \Bigr \vert, \sep \forall \eee_i \in \{\pm 1\},\  i=1,2.
	$$
	Define the mapping  $\Vert \ \Vert $ on $\R ^3$ by	
	$$
	\Vert (r,s,t) \Vert = \vert (r,s) \vert + \vert t\vert,
	$$
	that is an absolute and normalized norm on $ \R ^3.$  Let $X$ be $\R^3,$ endowed with the previous norm and the usual order.
	
	We consider the elements $x=\frac{1}{4} (3,-3,0)$ and $x^*= \frac{1}{3} (2,-2,3).$  We know that
	$$
	x\in S_X, \sem x^+= \frac{3}{4 }e_1 \sem \text{and} \sem  x^-= \frac{3}{4 }e_2 
	$$
	so $\Vert x^+ \Vert = \Vert x^-\Vert= \dfrac{3}{4}.$
	Now we check that $x^* \in B_{ X^*}.$ Assume that $(r,s,t ) \in B_X$ and so
	$$
	 \frac{2}{3} \bigl( \vert r \vert + \vert s \vert \bigr) + \vert t \vert \le 1.
	$$
	Hence 
	\begin{align*}
	x^*(r,s,t)  & = \frac{2}{3} (r-s) + t \\
	\nonumber
	&  \le     \frac{2}{3} (\vert r \vert + \vert s \vert  ) + \vert t  \vert \\
	\nonumber
	&  \le 1 .
	\end{align*}
	Since $\Vert e_3 \Vert =1,$ $\Vert x^* \Vert  \ge x^*(e_3)=1.$ As a consequence $x^* \in S_{X^*}$ and we clearly have 
	$$
	x^* (x)=1, \sem x^{*+} = \frac{1}{3} (2,0,3) \sem \text{and}   \sem x^{*-} = \frac{1}{3} (0,2,0) .
	$$
	Since $\Vert  x^{*+} \Vert \le  \Vert x^{*} \Vert =1$ and $x^{*+} (e_3)=1$ we have that $\Vert  x^{*+} \Vert=1.$ It is also satisfied that $\Vert  x^{*-} \Vert = \frac{2}{3}.$
	
	As a consequence we obtain that
	$$
	\bigl \Vert  x^{*+} \bigr \Vert  \,  \bigl \Vert  x^+ \bigr \Vert + \bigl \Vert  x^{*-} \bigr \Vert  \,\bigl \Vert  x^- \bigr \Vert  =1 \frac{3}{4} + \frac{2}{3}\frac{3}{4} = \frac{5}{4} > 1,
	$$
	so $X$ does not satisfy the hereditary norm attaining property.
\end{proof}

\begin{example}
\label{exam-dim-3-N-SM}
	There is a three  dimensional Banach lattice $X$ which is not strongly  monotone.
\end{example}
\begin{proof}
	We consider  the space $X= \R^3$ endowed with the usual order and  the absolute  norm whose closed unit ball is the convex hull of  the  following set 
	$$
	\{ (x,y,0): x^2 + y^2 \le 1\} \cup
	\{ (x,0,z): x^2 + z^2 \le 1\}
	$$
	$$
	\cup
	\{ (0,y,z): y^2 + z^2 \le 1\} \cup
	\Bigl\{ \frac{1}{ \sqrt{2}} (r,s,t): r,s,t \in \{1,-1\} \Bigr\}.
	$$
	
	Firstly  we will check that the dual norm of an element $x^*\in X^*$  is given by
	\begin{equation}
	\label{3-4-dual-norm}
\Vert x^* \Vert = \max \Bigl \{  \Vert P_{ ij} (x^*) \Vert _2, i, j \in \{1,2,3\}, \frac{\Vert x^* \Vert _1}{\sqrt{2}  } \Bigr \}, 
		\end{equation}
	
	where we identified the dual of $X$ as the set $\R^3$ and    $P_{ij} $  is the projection on $\R ^3 $ given by
	$$
	P_{ij} (x^*) = x^*(i) e_i +  x^*(j) e_j, \seg (1 \le i,j \le 3 ).
	$$
	Since 
	$$
	\{ (x,y,0): x^2 + y^2 \le 1\} \cup
	\{ (x,0,z): x^2 + z^2 \le 1\} \cup
	\{ (0,y,z): y^2 + z^2 \le 1\} \subset B_X,
	$$
	then it is immediate to check that 
	$$
	\Vert x^* \Vert \ge \max \bigl\{ \Vert P_{ij} (x^*) \Vert _2 : i,j \in \{1,2,3\} \bigr\}.
	$$
	From the inclusion
	$$
	\Bigl\{ \frac{1}{ \sqrt{2}} (r,s,t): r,s,t \in \{1,-1\} \Bigr\} \subset B_X,
	$$
	it follows that 
	$$
	\Vert  x^* \Vert \ge  \frac{ \Vert x^* \Vert _1}{\sqrt{2}}  .
	$$
	
	As a consequence we obtain that 
	$$
	\Vert x^* \Vert \ge  \max \Bigl \{  \Vert P_{ ij} (x^*) \Vert _2, i, j \in \{1,2,3\},  \frac{\Vert x^* \Vert _1}{\sqrt{2}  }\Bigr \},
	$$
	and the reverse inequality  is trivially satisfied.
	
	It is clear that for every $ 0 < r < 1$ the element
	$z= \dfrac{1}{2} \Bigl( (r, \sqrt{1-r^2}, 0) + \frac{1}{\sqrt{2}}  (1,1,-1)\Bigr) $  belongs to  $B_X.$
	Let choose a sequence $\bigl( r_n\bigr) $ of real numbers such that $0 < r_n< \dfrac{1}{\sqrt{2}}$  and $ ( r_n\bigr)  \to \dfrac{1}{\sqrt{2}}.$ 
	 Put $y= \dfrac{1}{\sqrt{2}}(1,1,-1)$ and consider for $n \in \N$ the element $z_n$ given by
	$$
	z_n = \frac{1}{2} \bigl( \bigl (r_n, \sqrt{1-r_{n}^2}, 0\bigr) +y\bigr) 
	$$
	that belongs to $B_X$ and satisfies
	$$
	z_{n} ^{+} = \frac{1}{2} \Bigl( r_n+ \frac{1}{\sqrt{2}},  \sqrt{1-r_{n}^2} +  \frac{1}{\sqrt{2}}, 0 \Bigl) \sem \text{and} \sem   z_{n} ^{-} = \frac{1}{2}  \Bigl(0,0,\frac{1}{\sqrt{2}}\Bigl)= \frac{1}{2 \sqrt{2}} e_3 .
	$$
	Since $\bigl( r_n\bigr) $ converges to $\dfrac{1}{\sqrt{2}}$	 we obtain that
	$$
	\lim \bigl(   z_n ^+  \bigr)  = \frac{1}{\sqrt{2}}  (1,1,0) .
	$$ 
	Since the element 	 $\frac{1}{\sqrt{2}} (1,1,0) $ belongs to $B_{X}$, the element $x^* = \frac{1}{\sqrt{2}}  \bigl( e_{1}^* + e_{2} ^*\bigr) \in B_{X^*}$ and 
	$$
	\frac{1}{\sqrt{2}}  \bigl( e_{1}^* + e_{2} ^*\bigr) \Bigl(\frac{1}{\sqrt{2}} (1,1,0) \Bigr) =1,
	$$
	then $\frac{1}{\sqrt{2}} (1,1,0)  \in S_X.$
	
	{\bf Claim.} For every $\alpha \in \R^*$ it is satisfied that
	$$
	\Bigl \Vert \frac{z_{n } ^+}{\Vert z_{n } ^+\Vert } + \alpha e_3 \Bigr \Vert > 1, \seg \forall n \in \N.
	$$
	We prove the claim.  Let us fix $n \in \N.$  Firstly let us notice  that $\Vert z_ n ^+ \Vert =  \Vert z_ n ^+ \Vert_2.$  Since $z_n (1) $ and $z_n (2) $ are positive and $z_n (1) \neq z_n (2), $ we conclude that   $\Vert	z_{n} ^{+} \Vert _1  <  \sqrt{2}  \Vert 	z_{n} ^{+} \Vert _2  .$  So we can choose  a real number $u_n$ such that 
	$$
	0 < u_n <  \min \{ z_n (1), z_n (2), 	
	\sqrt{2} \Vert 	z_{n} ^{+} \Vert _2 - 
	\Vert 	z_{n} ^{+} \Vert _1	\}.
	$$

	If we write
	$x_{n}^* = \Bigl( \frac{ z_n (1)}{ \Vert 	z_{n} ^{+} \Vert _2}, \frac{ z_n (2)}{ \Vert	z_{n} ^{+} \Vert _2}, \frac{ u_n }{ \Vert 	z_{n} ^{+} \Vert _2}\Bigr) = \Bigl( \frac{ z_n ^+(1)}{ \Vert 	z_{n} ^{+} \Vert _2}, \frac{ z_n^+ (2)}{ \Vert	z_{n} ^{+} \Vert _2}, \frac{ u_n }{ \Vert 	z_{n} ^{+} \Vert _2}\Bigr),   $   in view of \eqref{3-4-dual-norm} we have that
	$$
	\Vert x_{n}^* \Vert = \max \Bigl \{
	\frac{	\Vert 	z_{n} ^{+} \Vert _2 }{ \Vert 	z_{n} ^{+} \Vert _2 },
	\frac{1 }{ \Vert z_n ^+ \Vert _2 } \Vert (z_n(1), u_n ) \Vert _2, 
	\frac{1 }{ \Vert z_n ^+\Vert _2 } \Vert (z_n(2), u_n ) \Vert _2, 
	\frac{1} {\sqrt{2} \ \Vert z_n^+ \Vert _2 } \Vert (z_n(1), z_n (2), u_n ) \Vert _1 
	\Bigr \}.
	$$
Since  $0 < u_n < \min \{ z_n (1), z_n (2)\}$  and $z_n \in B_X, $the first three numbers in the previous expression are  less or equal to $1.$  By the choice of $u_n$ we also have  that $ \Vert z_n^+ \Vert _1 +  u_n < \sqrt{2} \Vert z_n^+ \Vert _2,$ so the last number in the expression of  $\Vert x_n ^* \Vert $ is also less or equal to $1.$ As a consequence
$\Vert x_n^* \Vert \le 1 $ and 
		$$
	x_{n}^* \Bigl( \frac{z_{n}^+}{\Vert    z_{n}^+ \Vert_2 }
	+ \alpha e_3 \Bigr) = 
	x_{n}^{*}
	\Bigl( \frac{z_{n}^{+} }{ \Vert z_{n} ^{+} \Vert_2 }  \Bigr) + 
	\frac {u_n }{\Vert z_n ^+ \Vert _2} \alpha  = 1 + \frac {u_n }{\Vert z_n ^+ \Vert _2} \alpha  > 1
	$$
	in case that $\alpha > 0.$ So we checked that 
	$$
	\Bigl \Vert \frac{z_{n } ^+}{\Vert z_{n } ^+\Vert_2 } + \alpha e_3 \Bigr \Vert > 1
	$$
	for $\alpha \in \R^+.$  Since the norm of $X$ is absolute and $z_n^+(3)=0,$ the same condition is also satisfied for $\alpha  < 0.$ 

	Since $\lim \bigl(  \Vert z_{n}^+\Vert \bigr) =1$ and $\Vert z_n ^- \Vert  = \frac{1}{2\sqrt{2}} $  for any natural number $n,$  we conclude that the space $X$  is not strongly  monotone.
\end{proof}

 In view of Proposition \ref{pro-finite-BPBpp},  Examples  \ref{exam-dim-3-N-HNAp}  and  \ref{exam-dim-3-N-SM}   show that none of the sufficient conditions in  Theorem \ref{teo-SM-HNA-BPBpp}  are   necessary conditions.

For  the next example we need some auxiliary results.

\begin{lemma}
	\label{le-strict-mon-abs}
	Let  $\vert \ \vert $ be an absolute and normalized norm on $\R ^2.$  Then
	$$
	0 \le s < t, \sep 0 \le  u< v \ \Rightarrow \ \vert (s,u) \vert < \vert (t,v) \vert.
	$$
\end{lemma}
\begin{proof}
	Let us choose $0 < \eee < \min \bigl  \{1- \frac{s}{t}, 1- \frac{u}{v}\bigr\}$ and so  
	$$
	s  < (1- \eee ) t \sep \text{and} \sep u < (1- \eee) v.
	$$
	As a consequence
	$$
	\vert (s,u) \vert \le \vert  (1- \eee) t, (1-\eee)v) \vert = (1- \eee) \vert (t,v) \vert < \vert (t,v)\vert.
	$$	
\end{proof}

\begin{lemma}
	\label{le-strict-mon-2}
	Let  $\vert \ \vert $ be an absolute and normalized norm on $\R ^2.$  Assume that $c,d,e \in \R ^+_{0},$ $d \ne e$ and $\vert (c,d) \vert =1 = \vert (c,e) \vert$. Then $c=1.$ 
\end{lemma}
\begin{proof}
	Since the norm $\vert \ \vert$  is absolute and normalized we have that
	$$
	c= \vert (c,0)\vert \le \vert (c,d)\vert  \le 1.
	$$
	Assume that $ d<e.$

	Assume that $c < 1.$   For any $d \le y \le   e$ we have that 
	$$
	1 = \vert (c,d)\vert \le \vert (c,y)\vert \le \vert (c,e)\vert=1,
	$$
	so
	$$
	\{(c,y): y \in [d,e] \} \subset S_{\R^2}.
	$$
	Since the unit ball is  convex, it contains the segment whose extreme  points are $(c,e)$ and  $(1,0).$ Since $0 < d < e$  there is a positive real number $x$ such that $(x,d)$ belongs to that segment and so $\vert (x,d) \vert \le 1.$ By using that the norm is absolute and $c < x$ we have that
	$$
	1 = \vert (c,d) \vert \le \vert (x,d)\vert   \le 1,
	$$
	so $ \vert (x,d) \vert  =1.$
	
	By using that $0 \le d < e$ and $c< x < 1$ the element
	$$
	z= \frac{1}{3} \bigl( (c,d) + (x,d) + (c,e) \bigr) 
	= \Bigl( \frac{ 2c+x}{3}, \frac{2d+e}{3}\Bigr),
	$$
	satisfies that
	$$
	1= \vert (c,d) \vert \le  \Bigl \vert   \Bigl( \frac{ 2c+x}{3}, \frac{2d+e}{3}\Bigr) \Bigr \vert \le 1,
	$$ 
	and so $u = \Bigl( \frac{ 2c+x}{3}, \frac{2d+e}{3}\Bigr)$ belongs to the unit  sphere.
	Since 
	$$
	c < \frac{2c+x}{3} \sem \text{and} \sem d < \frac{2d+e}{3},
	$$
	and $\vert (c,d) \vert =1$  the element $u$ does not belongs to the unit sphere  in view of Lemma \ref{le-strict-mon-abs}. Hence $c=1$ as we wanted to show.
\end{proof}

\begin{lemma}
	\label{le-strict-mon-3}
	Let  $\vert \ \vert $ be an absolute and normalized norm on $\R ^2.$  Assume that
	$$
	0 = \max\{ x \in \R :  \vert (x,1) \vert =1 \} = 
	\max\{ y \in \R :  \vert (1,y) \vert =1 \}.
	$$
	Then
	$$
	\vert s \vert < \vert t \vert \ \Rightarrow \ \vert (r,s) \vert < \vert (r,t) \vert  \sem \text{and} \sem 
	\vert (s,r) \vert < \vert (t,r) \vert.
	$$	  
\end{lemma}
\begin{proof}
Since the norm is absolute, it suffices to prove the statement  for $0 \leq  s  <  t .$ Notice that if $r=0 ,$  the assertion stated  in the lemma is satisfied. Otherwise,   by using again that the norm is absolute, it is enough to prove the statement for $ r>0 .$

	Since the norm is absolute  we know that  $\vert (r,s) \vert \le \vert (r,t) \vert .$  In case that   $\vert (r,s) \vert = \vert (r,t) \vert $  it is also satisfied 
	$$
	\Bigl \vert  \Bigl( \frac{r}{ \vert (r,s) \vert}, \frac{s}{\vert (r,s) \vert} \Bigr) \Bigr\vert = 1 = 
	\Bigl \vert  \Bigl( \frac{r}{ \vert (r,s) \vert}, \frac{t}{\vert (r,s) \vert} \Bigr) \Bigr\vert.
	$$
	By Lemma \ref{le-strict-mon-2} we obtain that $r= \vert (r,s) \vert,$ that is,
	$$
	1= \Bigl \vert  \Bigl(1,  \frac{t}{ r} \Bigr) \Bigr\vert.
	$$
	By assumption we get $t=0,$ which is impossible, so we proved that $\vert (r,s) \vert < \vert (r,t) \vert .$  From the same argument we also obtain that  $ \vert (s,r) \vert < \vert (t,r) \vert .$
\end{proof}

\begin{example}
	\label{ell2-not-WM}
	   There is a Banach lattice isomorphic to $\ell_2$ that is not weakly  monotone. So this space  does not have the Bishop-Phelps-Bollob{\'a}s property for positive functionals.
\end{example}
\begin{proof}
	Let $\bigl(  \alpha _n\bigr) $ be  a sequence of positive real numbers that is strictly increasing  and converges to $1.$ For every natural number $n,$ consider the absolute norm  $ \vert \ \vert _n$ on $\R ^2$  whose closed unit ball is  the set $B_n$  given by
	$$
	B_n = \co \Bigl\{  \pm e_1, \pm e_2, \Bigl( \alpha _n, \pm \frac{1}{2}\Bigr),  \Bigl( -\alpha _n, \pm \frac{1}{2}\Bigr)\Bigr\}.
	$$
	The norm $\vert \ \vert _n$ clearly satisfies the assumption of Lemma \ref{le-strict-mon-3}.
	
	We write  $X= \ell_2 $ as a Riesz space. Since $\vert \ \vert _n$ is an  absolute and normalized norm for each natural number $n$ we have that
	$$
	\max\{ \vert x_{2n-1} \vert,  \vert x_{2n} \vert \} \le  \vert  ( x_{2n-1} ,   x_{2n}  )\vert _n \le \vert x_{2n-1} \vert+   \vert x_{2n} \vert, \seg \forall n \in \N.
	$$
	So we can define the norm on $X$ given by
	$$
	\Vert x \Vert = \Vert \{  \vert  ( x_{2n-1}  ,   x_{2n}  )\vert _n \} \Vert _2, \seg  (x \in X).
	$$

	Clearly the previous norm is equivalent to the usual norm on  $\ell _2$ and it is a lattice norm on $X.$
	
	For a nonempty subset $C \subset \N$ we denote by $P_C$ the operator  $P_C : \ell_2 \llll \ell_2 $ given by
		$$
		P_C(x)= \sum _{ n \in C} x(n) e_n \seg (x \in \ell_2).
		$$
		In case that $C=\varnothing$ we simply take $P_C=0.$

	Since we can apply Lemma \ref{le-strict-mon-3} to the norm  $\vert \ \vert_k$ for any natural number $k$ we have that
	$$
	\vert ( x_{2k-1}, 0 ) \vert _k < \vert ( x_{2k-1}, x_{2k} ) \vert _k, \sem \forall x  \in X, x_{2k} \ne 0.
	$$
	Of course, the analogous inequality holds   changing the role of the coordinates. 
	Since the norm of $\ell_2$ is strictly monotone, for any natural number $n$  we conclude that
	$$
	\Vert P_{\N \backslash \{n\} } (x)           \Vert < \Vert x \Vert, \seg \forall x \in X, x_n \ne 0.
	$$
	Since $X$ is a Banach  lattice, as a consequence of the previous inequality, we  obtain that
	\begin{equation}
	\label{mon-X}
	x \in X, A \subset \N, P_A(x) \ne 0 \ \Rightarrow \ 
	\Vert (I-P_A)(x)) \Vert =  \Vert P_{\N \backslash A} (x) \Vert  < \Vert x \Vert .
	\end{equation}
	Notice that for $x \in X,$ if we denote  by $A$ and $N$ the sets
	$$
	A = \{ n \in \N: x_n  > 0\} \seg \text{and} \seg 
	B = \{ n \in \N: x_n  <  0\},
	$$
	 then we have that
	$$
	x^+ =P_A(x) \seg \text{and} \seg   x^- =-P_B(x).
	$$
	By taking into account  \eqref{mon-X} we deduce  that  
	\begin{equation}
	\label{x-x+-pos}
	x \in X, \ \Vert x \Vert = \Vert x^+ \Vert \ \Rightarrow \ x \ge 0.
	\end{equation}

	Finally, we show that $X$ is not weakly monotone. In order to do this, we define the sequence $\bigl( z_n\bigr) $ given by
	$$
	z_n(2n-1)= \alpha _n, \sep    z_n(2n)= - \frac{1}{2} \sep \text{and} \sep z_n(k)=0 \ \text{if } \ k \in \N \backslash \{2n-1, 2n\}.
	$$
	For each natural number $n,$  we  clearly have that $z_n \in X$  and  $\Vert z_n \Vert = \bigl\vert  \bigl (\alpha _n, -\frac{1}{2} \bigr) \bigr\vert_n =1.$ We also have that
	$
	z_{n}^+ = \alpha _n e_{2n-1}  $ and $z_{n}^- = \frac{1}{2} e_{2n} ,$ so
	$$
	\Vert z_{n}^+ \Vert  = \alpha _n \sem \text{and} \sem 
	\Vert z_{n}^- \Vert = \frac{1}{2}.
	$$
	If $x \in X$ and   $\Vert x \Vert = \Vert x^+ \Vert =1, $ by \eqref{x-x+-pos} we obtain that $x= x^+,$ that is, $x^-=0.$ Hence
	$$
	\frac{1}{2}  = \Vert z_{n} ^- \Vert  = \Vert x^-  - z_{n} ^-  \Vert \le \Vert x -  z_n \Vert, \sem \forall n \in  \N.
	$$
	Since $\lim \bigl(  \Vert z_{n} ^+ \Vert \bigr)  = \lim \bigl(  \alpha _n \bigl) =1,$ the space $X$ is not  weakly  monotone.  
	
	In view of Theorem \ref{teo-SM-HNA-BPBpp}, the space $X$ cannot  have the Bishop-Phelps-Bollob{\'a}s property for positive functionals.
\end{proof}

\bibliographystyle{amsalpha}

\end{document}